\documentclass[runningheads]{llncs}
\usepackage{graphicx}
\usepackage{amsmath}
\usepackage{amsfonts}
\begin{document}

% special things for me

%\newcommand{\N}{\bbbn}
\newcommand{\N}{\mathbb N}
\newcommand{\Z}{\mathbb Z}

\title{Maximality of reversible gate sets\thanks{The research has been supported by Austrian Science Fund (FWF) research projects AR561 and P29931.}}
\author{Tim Boykett\inst{1,2,3}}
\authorrunning{T. Boykett}
\institute{Institute for Algebra, Johannes-Kepler University, Linz, Austria \and
Time's Up Research, Linz, Austria
\email{tim@timesup.org}\\
\url{http://timesup.org/cca} \and
University for Applied Arts, Vienna}
\maketitle              % typeset the header of the contribution
\begin{abstract}
In order to better understand the structure of closed collections of reversible gates, 
we investigate the lattice of closed sets and the maximal members of this lattice.
In this note, we find the maximal closed sets over a finite alphabet. 
We find that for odd sized alphabets,
there are a finite number of maximal closed sets, while for the even case we have a countable
infinity, almost all of which are related to an alternating permutations.
We then extend to other forms of closure for 
reversible gates, ancilla and borrow closure.
Here we find some structural results, including some examples of maximal closed sets.

%\keywords{Reversible gates  \and Maximal closed classes \and permutation groups}
\end{abstract}

\section{Introduction}

For a given finite set $A$, we investigate the collections of reversible gates, or
bijections of $A^k$ for all $k$.
In many senses, these are akin to a bijective, permutative or reversible clone. 
The work derived from Tomasso Toffoli's work \cite{toffoli80} and as such we call closed systems of bijections 
reversible Toffoli Algebras (RTAs).
The work also relates  to permutation group theory, as an RTA $C$ is a $\N$-indexed collection of permutations
groups, $C^{[i]} \leq Sym(A^i)$.

In previous papers, Emil Je\v{r}\'{a}bek has found the dual structure \cite{jerabek18} and the author, together with Jarkko Kari and Ville Salo, 
has investigated generating sets \cite{b19,bks17} and other themes.

In this paper, we determine the possibile maximal closed systems, relying strongly on Liebeck, Praeger and Saxl's work\cite{LPS87}, and determine some properties of maximal borrow and ancilla closed RTAs.

We show that the maximal RTAs are defined by an index that defines the single arity at which the RTA is not the full set of bijections.
We then show that for different indices and orders of $A$, only certain possibilities can arise.
For ancilla and borrow closed RTAs we find that there is similarly an index below which the maximal RTAs are full symmetry groups and above which they are never full.

We start by introducing the background properties of RTAs and some permutation group theory.
The rest of the paper is an investigation of maximality, with the main result, Theorem \ref{theoremmax} taking up the main body of this section. We then investigate properties of borrow and ancilla closed RTAs.

\section{Background}

In this section we will introduce the necessary terminology.

Let $A$ be a finite set. $Sym(A)=S_A$ is the set of permutations or bijections of $A$, $Alt(A)$ the set of permutations of even parity.
If $A=\{1,\dots,n\}$ we will write $S_n$ and $A_n$.
We write permutations in cycle notation and act from the right.
We write the action of a permutation $g\in G \leq Sym(A)$ on an element $a\in A$ as $a^g$.
A subgroup $G\leq S_A$ is \emph{transitive} if for all $a,b\in A$ there is a $g\in G$ such that $a^g=b$.
We also say that $G$ acts \emph{transitively} on $A$.
A subgroup $G$ of $S_A$ acts \emph{imprimitively} if there is a nontrivial equivalence relation $\rho$ on $A$ such that
for all $a,b\in A$, for all $g\in G$, $a \rho b \Rightarrow a^g \rho b^g$.
If there is no such equivalence relation, then $G$ acts \emph{primitively} on $A$.

Let $G$ be a group of permutations of a set $A$.
Let $n \in \N$.
Then the \emph{wreath product} $G wr S_n$ is a group of permutations acting on $A^n$.
The elements of $G wr S_n$ are $\{(g_1,\dots,g_n,\alpha) \mid g_i \in G,\,\alpha \in S_n\}$ with action defined as follows:
for $(a_1,\dots,a_n)\in A^n$,  $(a_1,\dots,a_n)^{(g_1,\dots,g_n,\alpha)} 
= (a_{\alpha^{-1}1}^{g_1},\dots,a_{\alpha^{-1}n}^{g_n})$.

Let $B_n(A) = Sym(A^n)$ and $B(A) = \bigcup_{n\in \N} B_n(A)$. We call $B_n(A)$ the set of $n$-ary reversible gates on $A$,
$B(A)$ the set of reversible gates. 
For $\alpha \in S_n$, let $\pi_\alpha \in B_n(A)$ be defined by $\pi_\alpha(x_1,\dots,x_n) = (x_{\alpha^{-1}(1)},\dots,x_{\alpha^{-1}(n)})$.
We call this a \emph{wire permutation}.
Let $\Pi=\{\pi_\alpha \vert \alpha \in S_n,\,n\in\N\}$.
In the case that $\alpha$ is the identity, we write $i_n=\pi_\alpha$, the $n$-ary identity.
Let $f\in B_n(A)$, $g\in B_m(A)$. Define the \emph{parallel composition} as
$f\oplus g \in B_{n+m}(A)$ with $(f\oplus g)(x_1,\dots,x_{n+m}) = (f_1(x_1,\dots,x_n),\dots,f_n(x_1,\dots,x_n),g_1(x_{n+1},\dots,x_{n+m}),\dots,g_m(x_{n+1},\dots,x_{n+m}))$.
For $f,g\in B_n(A)$ we can compose $f\bullet g$ in $Sym(A^n)$. If they have distinct arities we ``pad'' them with identity, for instance $f\in B_n(A)$ and
$g\in B_m(A)$, $n< m$, then define $f\bullet g = (f \oplus i_{m-n}) \bullet g$ and we can thus serially compose all elements of $B(A)$.

We call a subset $C \subseteq B(A)$ that includes $\Pi$ 
and is closed under $\oplus$
and $\bullet$ a \emph{reversible Toffoli algebra} (RTA)
based upon Toffoli's original work \cite{toffoli80}.
These have also been investigated as \emph{permutation clones} \cite{jerabek18} and with ideas from category theory \cite{lafont93}.
If we do not insist upon the inclusion of $\Pi$, the get \emph{reversible iterative algebras}  \cite{bks17} in reference to Malcev and Post's iterative algebras.
For a set $F \subseteq B(A)$ we write $\langle F \rangle$ as the smallest RTA that includes $F$, the RTA \emph{generated} by $F$.

Let $q$ be a prime power, $GF(q)$ the field of order $q$, 
$AGL_n(q)$ the collection of affine invertible maps of $GF(q)^n$ to itself.
Let $\operatorname{Aff}(q) = \bigcup_{n\in \N} AGL_n(q)$ be the RTA of affine maps over $A=GF(q)$.

Let $C$ be an RTA. We write $C^{[n]} = C \cap B_n(A)$ for the elements of $C$ of arity $n$.
We will occasionally write $(a_1,\dots,a_n)\in A^n$ as $a_1a_2\dots a_n$ for clarity.

We say that an RTA $C \leq B(A)$ is \emph{borrow closed} if for all $f\in B(A)$, 
$f\oplus i_1 \in C$ implies that $f\in C$.
We say that an RTA $C \leq B(A)$ is \emph{ancilla closed} if for all $f\in B_n(A)$, 
$g \in C^{[n+1]}$ with some $a\in A$ such that for all $x_1,\dots,x_n \in A$, for all
$i\in\{1,\dots,n\}$, $f_i(x_1,\dots,x_n) = g_i(x_1,\dots,x_n,a)$ and
$g_{n+1}(x_1,\dots,x_n,a)=a$ implies that $f\in C$.
If an RTA is ancilla closed then it is borrow closed.

\section{Permutation Group Theory}
In this section we introduce some results from permutation group theory that will be of use.
The maximal subgroups of permutation groups have been determined.

\begin{theorem}[\cite{LPS87}]
\label{maxsubgp}
 Let  $n \in \N$. Then the maximal subgroups of $S_n$ are conjugate to one of the following $G$.
 \begin{enumerate}
 \item (alternating) $G=A_n$
  \item (intransitive) $G=S_k \times S_m$ where $k+m=n$ and $k\neq m$
  \item (imprimitive) $G=S_m wr S_k$ where $n=mk$, $m,k > 1$ 
  \item (affine) $G=AGL_k(p)$ where $n=p^k$, $p$ a prime
  \item  (diagonal) $G= T^k.(Out(T) \times S_k)$ where $T$ is a nonabelian simple group, $k>1$ and $n = \vert T \vert ^{(k-1)}$
  \item (wreath) $G= S_m wr S_k$ with $n=m^k$, $m\geq 5$, $k>1$
  \item (almost simple) $ T \triangleleft G \leq Aut(T)$, $T\neq A_n$ a nonabelian simple group, $G$ acting primitively on $A$
 \end{enumerate}
 Moreover, all subgroups of these types are maximal when they do not lie in $A_n$, except for a list of known exceptions.
\end{theorem}

It is worth noting that in the imprimitive case, we have an equivalence relation with $k$ equivalence clases of order $m$,
the wreath product acts by reordering the equivalence classes, then acting an $S_m$ on each equivalence class.
In the wreath case, the set $A$ is a direct product on $k$ copies of a set of order $m$, the wreath product acts by
permuting indices and then acting as $S_m$ on each index.

\begin{lemma}
\label{lemmaeven3}
 Let $A$ be a set of even order and $n\geq 3$. Then $S_A wr S_n \leq Alt(A^n)$.
\end{lemma}
\begin{proof}
 $S_A wr S_n$ is generated by $S_A$ acting on the first coordinate of $A^n$ and $S_n$ acting on coordinates.
 
 The action of $S_A$ on $A^n$ is even because for each cycle in the first coordinate, the remaining $n-1$ coordinates are untouched.
 So every cycle occurs $\vert A \vert^{n-1}$ times, which is even, so the action of $S_A$ lies in $Alt(A^n)$.
 
 $S_n$ is generated by $S_{n-1}$ and the involution $(n-1\,n)$. By the same argument, each cycle of the action occurs an 
 even number of times, so the action of $S_{n-1}$ and the involution $((n-1)\,n)$ on $A^n$ lies in $Alt(A^n)$ so we are done.\qed
\end{proof}
We have a similar inclusion for affineness.

\begin{lemma}
\label{lemmaaffineeven}
 For $n\geq 3$, $AGL_n(2)\leq Alt(2^n)$. 
\end{lemma}
\begin{proof}
$AGL_n(2)$ is generated by the permutation matrices  
$\{\pi_{(1,i)} \mid i=2,\dots,n\}$ and the matrix $\begin{bmatrix} 1 & 1\\0&1\end{bmatrix} \oplus i_{n-2}$.
These bijections are even parity because they only act on  two entries, thus 
have parity divible by $2^{n-2}$ modulo 2 which is 0.\qed
\end{proof}

%%CITE the result from Schneider, Praeger. (Which one?)

\begin{lemma}
\label{lemmaeven4}
 Let $A$ be a set of even order. Then $S_A wr S_2 \leq Alt(A^2)$ iff $4$ divides $\vert A\vert$.
\end{lemma}
\begin{proof}
 The same argument as above applies for $S_A$.
 The action of $S_2$ swaps $\frac{\vert A\vert (\vert A\vert-1)}{2}$ pairs. This
 is even iff $4$ divides $\vert A\vert$.\qed
\end{proof}

\section{Maximality in RTAs}
In this section, we will determine the maximal RTAs on a finite set $A$.

We have some nice generation results from other papers that will be useful.

\begin{theorem}[\cite{b19} Theorem 5.9]
\label{theoremAodd}
 Let $A$ be odd. If $B_1(A),B_2(A) \subseteq C \subseteq B(A)$, then $C=B(A)$.
\end{theorem}

\begin{theorem}[\cite{bks17} Theorem 20]
\label{theoremAeven_BKS}
 If $Alt(A^4) \subseteq C \subseteq B(A)$ then $Alt(A^k) \subseteq C$ for all $k\geq 5$.
\end{theorem}

The techniques used in the proof of Theorem \ref{theoremAeven_BKS} can be adapted to prove some further results.

\begin{lemma}
\label{lemma3trans}
 Let $\vert A\vert \geq 3$, then $\langle B_1(A),B_2(A)\rangle$ is 3-transitive on $A^3$.
\end{lemma}
\begin{proof}
 Let $A=\{1,2,3,\dots\}$. Let $a,b,c \in A^3$ be distinct. We show that we can map these to $111, 112,113 \in A^3$.
 
 Suppose $a_3,b_3,c_3$ all distinct.
 Let $\alpha = (a_1a_3\;1a_3)(b_1b_3\;1b_3)(c_1c_3\;1c_3) \in B_2(A)$.
 Let $\beta = (a_2a_3\;11)(b_2b_3\;12)(c_2c_3\;13) \in B_2(A)$.
 Then $\gamma = (\pi_{(23)} \bullet (\alpha \oplus i_1) \bullet \pi_{(23)}) \bullet (i_1\oplus \beta)$ satisfies the requirements.
 
 Suppose $a_3,b_3,c_3$ contains two values, wlog suppose $a_3=b_3$.
 Let $d \in A - \{a_3,c_3\}$.
 Let $\delta = (a_1a_3\;a_1d) \in B_2(A)$.
 Let $\lambda = \pi_{((23)} \bullet (\delta \oplus i_1) \bullet \pi_{(23)}$.
 Then $\lambda$ will map $a,b,c$ to the situation is the first case.
 
 Suppose $a_3=b_3=c_3$. Then one of $a_1,b_2,c_1$ or $a_2,b_2,c_2$ must contain at least two values, wlog let $a_1,b_2,c_1$ be so.
 Then $\pi_{(13)}$ will give us the first case if it contains three values, the second case if it contains two cases.\qed
\end{proof}

The two following results are only relevant for even $A$.

\begin{lemma}
 \label{lemmaAeven_M3}
 Let $\vert A\vert \geq 4$, $B_1(A),B_2(A) \subset C \leq B(A)$. Then $Alt(A^3) \subseteq C^{[3]}$.
\end{lemma}
\begin{proof}
 For $\vert A \vert = 4$, the result is shown by calculation in GAP \cite{GAP4} that the orders agree. 
 
 For $\vert A \vert = 5$ the result follows from Theorem \ref{theoremAodd}.

 Suppose $\vert A \vert \geq 6$ Since $B_2(A) \subseteq C$, we have all 1-controlled permutations of $A$ in $C$.
 By \cite{bks17} Lemma 18, with $P\subset Alt(A)$ the set of all 3-cycles, we have all 2-controlled 3-cycles in $C$.
 Thus $(111\; 112\;113) \in C$.
  $B_1(A)\cup B_2(A)$ is 3-transitive on $A^3$ by Lemma \ref{lemma3trans}, so we have all 3-cycles in $C$, 
 so $Alt(A^3) \subseteq C$.\qed
\end{proof}

We know that this is not true for $A$ of order 2, where $B_2(A)$ generates a group
of order 1344 in $B_3(A)$, which is of index 15 in $Alt(A^3)$ and is included in no
other subgroup of $B_3(A)$. However we find the following.

\begin{lemma}
 \label{lemmaAeven_M4}
 Let $\vert A\vert$ be even, $B_1(A),B_2(A),B_3(A) \subset C \leq B(A)$. Then $Alt(A^4) \subseteq C^{[4]}$.
\end{lemma}
\begin{proof}
 For $A$ of order 4 or more, we use the same techniques as in Lemma \ref{lemmaAeven_M3}.
 
 For $A$ of order 2, we calculate.
 We look at $C^{[4]}$ as a subgroup of $S_{16}$.
 The wire permutations are generated by permutations 
 $(2,9,5,3)(4,10,13,7)(6,11)(8,12,14,15)$ and $(5,9)(6,10)(7,11)(8,12)$.
 $B_3(A)$ as a subgroup of $B_4(A)$ acting on the indices 2,3,4 is generated
 by
 $(1,2,3,4,5,6,7,8)(9,10,11,12,13,14,15,16)$ and $(1,2)(9,10)$.
 It is a simple calculation to determine that this group is the entire alternating group $A_{16}$, so $Alt(A^4)\subseteq C^{[4]}$. \qed
\end{proof}

We can now state our main theorem.

\begin{theorem}
\label{theoremmax}
 Let  $A$ be a finite set. Let $M$ be a maximal sub RTA of $B(A)$.
 Then $M^{[i]} \neq B_i(A)$ for exactly one $i$ and $M$ belongs to the following classes:
  \begin{enumerate}
   \item $i=1$ and $M^{[1]}$ is one of the classes in Theorem \ref{maxsubgp}.
   \item $i=2$, $\vert A \vert =3$,  and  $M^{[2]}= AGL_2(3)$  (up to conjugacy)
   \item $i=2$, $\vert A \vert \geq 5$ is odd  and $M^{[2]}= S_A wr S_2$ % is this actually the same as AGL_2(3)
   \item $i=2$, $\vert A \vert \equiv 2 \mod 4$  and $M^{[2]}=S_A wr S_2$ 
   \item $i=2$, $\vert A \vert \equiv 0 \mod 4$  and $M^{[2]}=Alt(A^2)$ 
   \item $i=2$, $\vert A \vert \equiv 0 \mod 4$  and $M^{[2]} = T^{(3)}.(Out(T) \times S_{3})$ where $T$ is a finite nonabelian simple group, with $\vert A \vert = \vert T \vert$ (up to conjugacy)
   \item $i=2$, $\vert A \vert \equiv 0 \mod 4$  and  $M^{[2]}$ is an almost simple group (up to conjugacy)
   \item $i\geq 3$, $\vert A \vert$ is even  and $M^{[i]}  = Alt(A^i)$ 
  \end{enumerate}
\end{theorem}
\begin{proof}
 Suppose $M < B(A)$ with $i\neq j$ natural numbers such that $M^{[i]} \neq B_i(A)$ and  $M^{[j]} \neq B_j(A)$.
 Wlog, $i< j$, let $N = \langle M \cup B_j(A) \rangle$. 
 Remember that compositions of mappings of arity at least $j$ will also be of arity at least $j$, so $N^{[k]}=M^{[k]}$ for all  $k<j$.
 Then $M < N$ because $N$ contains all of $B_j(A)$ and $N < B(A)$ because $N^{[i]} = M^{[i]} \neq B_i(A)$. Thus $M$ was not maximal, proving our first claim.
 
 For the rest of the proof, take $M$ maximal with $M^{[i]} \neq B_i(A)$.
 Then $M^{[i]}$ is a maximal subgroup of $B_i(A)$.
 
 Suppose $i=1$. Then $B_1(A)= S_A$ and we are interested in the maximal subgroups of $S_A$. From Theorem \ref{maxsubgp} we know that these are in one of the 7 classes.
 
 Suppose $i \geq 2$. Then $S_A^i \leq M^{[i]}$ so $M^{[i]}$ is transitive on $A^i$.
 As $\Pi^{[i]} \leq M^{[i]}$ we also know that $S_A wr S_i \leq M^{[i]}$. 
 Assume $M^{[i]}$ acts imprimitively on $A^i$ with equivalence relation $\rho$.
 Let $a,b\in A^i$, ${a} \rho {b}$ with $a_i \neq b_i$. 
 By the action of $S_A$ acting on the $i$th coordinate we obtain ${a'} \rho {b'}$
 with $a_j=a_j'$ and $b_j = b_j'$ for all $j\neq i$. 
 By the action of $S_i$ on coordinates we can move this inequality to any index.
 Thus by transitivity we can show that $\rho = (A^n)^2$ and is thus trivial, so our action cannot be imprimitive.
 
 We now consider the cases of $A$ odd and even separately.
 
 Suppose $i\geq 2$ and $\vert A \vert$ is odd. If $i\geq 3$ then 
 $M^{[1]}=B_1(A)$ and $M^{[2]}=B_2(A)$,
 %we have all 1-ary and 2-ary mappings in $M$, 
 %thus all permutations of $A$ and the generalised binary Toffoli mappings. Thus 
 so by Theorem \ref{theoremAodd} we have all of $B(A)$ and thus $M$ is not maximal, a contradiction.
 Thus we have $i=2$. 
 $M^{[1]} = B_1(A)=S_A$ and $\pi_{(1\;2)} \in M$ so $M$ contains $S_A wr S_2$. 
 If $\vert A \vert \geq 5$ then by  Theorem \ref{maxsubgp} this is maximal
 in $Sym(A^2)$ so $M^{[2]}$ must be precisely this.  
 So the case of $A$ order 3 is left. 
 We want to know which maximal subgroups of $Sym(A^2)$ contain $S_A wr S_2$.
 There are 7 classes of maximal subgroups, we deal with them in turn.
 \begin{itemize}
  \item  Since $\pi_{(1\,2)} \in M$ is odd on $A^2$,  $M^{[2]} \not \subseteq Alt(A^2)$.
  \item From the discussion above we know that $M^{[2]}$ is transitive and  primitive on $A^2$, so the second and third cases do not apply.
  \item The permutations in $S_3$ can be written as affine maps in $\Z_3$ and $\pi_{(1\,2)}$ can be written as $\begin{bmatrix}0 & 1\\ 1 & 0\end{bmatrix}$, 
  the off diagonal $2 \times 2$ matrix over $\Z_3$, so 
  $S_3 wr S_2$ embeds in the affine general linear group. Thus $M^{[2]} = AGL_2(3)$ is one possibility.
  \item The diagonal case requires $\vert T \vert^{k-1} = 9$ for some nonabelian finite simple group $T$, a contradiction.
  \item The wreath case requires $9 \geq 5^2$, a contradiction.
  \item By \cite{BUEKENHOUT1996215} all $G$ acting primitively on $A^2$ with subgroups that are nonabelian finite simple groups are subgroups
  of $Alt(A^2)$, and we have odd elements in $M$, so this is a contradiction.
 \end{itemize}
Thus the only maximal subgroup is $M^{[2]} = AGL_2(3)$.
 
 Suppose $i\geq 2$ and $\vert A \vert$ is even.
 We know from Theorem \ref{theoremAeven_BKS} that for $i>4$ we can get all of $Alt(A^i)$ 
 from $\cup_{1\leq j < i} B_j(A)$. 
 $Alt(A^i)$ is maximal in $Sym(A^i)$ so we are done.

%  For $2\leq i \leq 4$ and $\vert A\vert \geq 5$, Theorem  \ref{maxsubgp} seems to say that $S_A wr S_i$ is  maximal, unless it lies in  $Alt(A^i)$, 
%  which it does for $i \geq 3$ and for $i=2$ when $4 | \vert A\vert$.
%  This leaves $i=2$ and $\vert A \vert \sim 2 \mod 4$ with $S_A wr S_i$ a maximal subgroup.
 
 Thus we are left with 3 cases, $i=2,3,4$. 
 
  From Lemma \ref{lemmaAeven_M4} we know that for $i=4$ , $M^{[4]}=Alt(A^4)$ is
  the only possibility.

 From Lemma \ref{lemmaAeven_M3} we know that for $i=3$ and $\vert A \vert \neq 2$, $M^{[3]}=Alt(A^3)$ is
  the only possibility.
  For $\vert A \vert = 2$ we find that $B_2(A)$ generates a subgroup of $B_3(A)$ that is only included in $Alt(A^3)$, so again $M^{[3]}=Alt(A^3)$ is
  the only possibility.
 
 Thus we are left with the case $i=2$.
 From the above we know that the intransitive and imperfect cases cannot arise.
Thus we need to consider the wreath, affine, diagonal and almost simple cases.
 \begin{itemize}
   \item $\vert A \vert =2$: $S_A wr S_2$ has order 8, $B_2(2)$ has order 24, so $M^{[2]}=S_A wr S_2$
   is maximal and we are done.
   \item Case $6 \leq \vert A \vert \equiv 2 \mod 4$: Lemma \ref{lemmaeven4} above says that $S_A wr S_2 \not \leq Alt(A^2)$ so it is 
   maximal by Theorem \ref{maxsubgp}. 
   \item Case $\vert A \vert =4$: Alternating is possible by inclusion. 
   The affine case $AGL_4(2)$ lies in $A_{16}$ by Lemma \ref{lemmaaffineeven}.
     Diagonal not possible by order. Almost simple not possible because all primitive groups of degree 16 lie in the alternating group $A_{16}$ \cite{BUEKENHOUT1996215} .
   \item Case $8 \leq \vert A \vert \equiv0 \mod 4$: 
    Alternating is always possible. If $A=2^m$ for some $m$, then $AGL_m(2)$ might be possible, but lies in $Alt(A^2)$ by Lemma \ref{lemmaaffineeven}. 
    Diagonal, almost simple might be possible, if $S_A wr S_2 \leq M^{[2]}$.
 \end{itemize}
\qed
\end{proof}

\begin{corollary}
 For $A$ of odd order, there is a finite number of maximal RTAs.
\end{corollary}

\subsection{Dealing with the Almost simple and Diagonal cases}

We are left with the possibility for $A$ of order a multiple of 4, that the diagonal or almost simple
case can arise.

The diagonal case with $A$ of order equal to the order of a finite simple nonabelian group start with order 60.
The other possibility is that $\vert A\vert^2 =\vert T \vert$ for some finite simple nonabelian group $T$.
The only known result in this direction is in \cite{nsw80} where they show that 
symplectic groups  $Sp(4,p)$ where $p$ is a certain type of prime, now known as NSW primes, have square order.
The first of these groups is of order $(2^4 \cdot 3\cdot 5\cdot 7^2)^2$ 
corresponding to $A$ of order $(2^4\cdot 3\cdot 5\cdot 7^2)=11760$.
We note that the sporadic simple groups have order that always contains a prime to the power one, so they are not of square order.
We know that the Alternating group can never have order that is a square, 
as the highest prime less than $n$ will occur exactly once in the order of the group.
It might be possible that there are other finite simple groups of square order. As far as we are aware, there have been no further results in this direction.

Each of these possibilities is far beyond the expected useful arities for computational processes.

The other case is to look at almost simple groups.
In the above we saw that all primitive actions of degree $4^2$ are alternating. 
In order to carry one,
  we can hope to use results about primitive permutation groups of prime power \cite{CaiZhangPrimePowerDegree}
 and product of two prime power \cite{LiLiPrimitiveTwoPrimePowers} degrees. With these
 we should be able to determine almost simple examples up to degree 30. Once agin this
 would include all examples of arities expected to be useful for computational processes.

\section{Borrow and Ancilla closure}

The strength of Theorem \ref{theoremmax} is partially due to the fact that there is no
effect of the existence of mappings of a certain arity in a given RTA on the size of the lower arity part, as
there are no operators to lower the arity of a mapping. 
This does not apply with ancilla and borrow closure. 
In this section we collect some  results about maximal ancilla and borrow closed RTAs.
The following result reflects the first part of Theorem \ref{theoremmax}.

\begin{lemma}
\label{lemmamaxba_closure}
 Let $M \leq B(A)$ be a maximal borrow or ancilla closed RTA.
 Then there exists  some $k\in \N$ such that for all $i< k$, $M^{[i]}=B_i(A)$
 and for all $i \geq k$, $M^{[i]} \neq B_i(A)$.
\end{lemma}
\begin{proof}
 Suppose $M^{[k]}=B_k(A)$. Then for all $f\in B_m(A)$, $m<k$, $f\oplus i_{k-m} \in M$ so
 $f\in M$, so $M^{[m]}=B_m(A)$ for all $m\leq k$.
 As $M$ is maximal, there must be a largest $k$ for which $M^{[k]}=B_k(A)$, since otherwise 
 $M=B(A)$.\qed
\end{proof}

From Theorem \ref{theoremAodd} we then note the following.
\begin{lemma}
\label{lemma_ba_Aodd}
 Let $\vert A\vert$ be odd. Then $k =1,2$ are the only options in  Lemma \ref{lemmamaxba_closure}.
\end{lemma}
In this case, we can say a bit more in case $k=2$.
If $A$ is of order 3, then by the argument in Theorem \ref{theoremmax} above, we find that $M = \operatorname{Aff}(A)$, the
affine maps over a field of order 3.
Otherwise $A$ is at least 5 and $B_1(A)$ is no longer affine. See the result below.

Similarly, we obtain the following, but see Corollary \ref{cor_lemma_ba_Aeven4} below.
\begin{lemma}
\label{lemma_ba_Aeven4}
 Let $\vert A\vert \geq 4$ be even. Then in Lemma \ref{lemmamaxba_closure}, $k=1,2,3$ are the only options and for $i>k$, $M^{[i]}\neq Alt(A^i)$.
\end{lemma}
\begin{proof}
  We start by noting that for even $\vert A\vert$, for all $f\in B_i(A)$, 
$f \oplus i_1 \in Alt(A^{i+1})$. Thus if $M^{[i]}= Alt(A^i)$ for some $i>k$, 
then $M^{[i-1]}=B_i(A)$ which is a contradiction, which shows the second part of
the result.

Suppose $k \geq 4$, so $B_1(A),B_2(A),B_3(A) \subseteq M$. Then by Lemma  \ref{lemmaAeven_M4} 
$Alt(A^4) \subseteq M$, so by Theorem \ref{theoremAeven_BKS} $Alt(A^j)\subseteq M$ for all $j\geq 5$.
But we know that by borrow closure, this implies that $B_{j-1}(A) \subseteq M$ 
so $M$ is in fact $B(A)$.
This is a contradiction, so $k< 4$.\qed
\end{proof}

Using the same arguments, we obtain the following.
\begin{lemma}
 Let $\vert A\vert =2$. Then in  Lemma \ref{lemmamaxba_closure}, $k=1,2,3$ are the only options and for $i>k$, $M^{[i]}\neq Alt(A^i)$.
\end{lemma}
\begin{proof}
 Suppose $M$ is maximal with $k \geq 5$.
 Then by Theorem \ref{theoremAeven_BKS} we obtain $M^{[i]}= Alt(A^i)$ for all $i\geq 5$,
 which by the first argument in the previous Lemma, implies that $M$ is not maximal.
 
 Suppose $M$ is maximal with $k=4$. 
 We know that $M^{[3]}=B_3(A)$.
 Then by Lemma \ref{lemmaAeven_M4} we find that $M^{[4]} = Alt(A^4)$,
 so by Theorem \ref{theoremAeven_BKS} we obtain all of $Alt(A^5)$ so by borrow closure
 all of $B_4(A)$ and thus $M$ is not maximal.\qed
\end{proof}

In general we obtain some examples of maximal borrow and ancilla closed RTA.

\begin {lemma}
 For $\vert A\vert \geq 5$, the degenerate RTA $Deg(A)$ generated by $B_1(A)$ is a maximal borrow or ancilla closed RTA.
\end {lemma}
\begin{proof}
 Let $Deg(A)$ be generated by $B_1(A)=S_A$.
 Then $Deg(A)^{[i]}= S_A wr S_i$ for all $i\geq 2$ which is maximal in $B_i(A)$ by Theorem \ref{maxsubgp}.
 Thus any RTA $N$ properly containing $Deg(A)$ will have  $N^{[i]}=B_i(A)$
 for some $i\geq 2$ and thus $N^{[2]}=B_2(A)$ by Lemma \ref{lemmamaxba_closure}. 
 Let $f \in N^{[2]} - Deg(A)^{[2]}$, then $f\oplus f \in N{[4]} - Deg(A)^{[4]}$ so
 $N^{[4]}=B_4(A)$ and thus by Lemmas \ref{lemma_ba_Aodd} and \ref{lemma_ba_Aeven4},
 $N=B(A)$, so $Deg(A)$ is maximal.\qed
\end{proof}

\begin{corollary}
\label{cor_lemma_ba_Aeven4}
Let $\vert A\vert \geq 4$ be even. Then in  Lemma \ref{lemmamaxba_closure}, $k=1,2$ are the only options
\end{corollary}
\begin{proof}
From  Lemma \ref{lemma_ba_Aeven4} we know $k=1,2,3$ are possible. Suppose $M$ is maximal
in B(A) with $k=3$.

Suppose $\vert A\vert = 4$. 
$B_2(A)$ can be embedded in $B_4(A)$ represented on $S_{256}$ with the tuples in $A^4$
represented by the integers $1,\dots,256$, generated by the permutations
\begin{align*}
 &(1,2,3,4,5,6,7,8,9,10,11,12,13,14,15,16)\\
&\hspace{6mm}(17,18,19,20,21,22,23,24,25,26,27,28,29,30,31,32)\dots\\
&\hspace{6mm}\dots(241,242,243,244,245,246,247,248,249,250,251,252,253,254,255,256)
\end{align*}
and
$(1,2)
(17,18)\dots(241,242)$. With the wire permutations we obtain a subgroup of $S_{256}$
that is the alternating group, so $M^{[4]}=Alt(A^4)$ and by Theorem \ref{theoremAeven_BKS}
we then get $M^{[5]}=Alt(A^5)$ and thus $M$ is not maximal.

 Suppose $A$ is even with more than 6 elements.
 The degenerate RTA $Deg(A) \leq M$ because $M^{[1]}=B_1(A)$, but because $Deg(A)$
 is maximal and $M^{[2]}$ is a supergroup of $Deg(A)^{[2]}$, $M$ is all of $B(A)$ and is not maximal.\qed
\end{proof}

\begin{lemma}
 Let $A$ be  of prime power order. Then $\operatorname{Aff}(A)$ is a maximal borrow closed RTA and a maximal ancilla closed RTA.
\end{lemma}
\begin{proof}
 Let $M=\operatorname{Aff}(A)$.
 For every $i$, $M^{[i]}$ is maximal in $B_i(A)$ by Theorem \ref{maxsubgp}.
 Suppose $M$ is not maximal, so $M < N < B(A)$.
 
 Let $f\in B_n(A)$, $f\in N-M$. Then $N^{[n]}=B_n(A)$ by subgroup maximality, so for all $i<n$, $N^{[i]}=B_i(A)$.
 For all $j$, $f\oplus i_j  \in (N-M)^{[n+j]}$ so similarly $N^{[n+j]}=B_{n+j}(A)$ so $N=B(A)$ and $M$ is maximal.
 
 Because $\operatorname{Aff}(A)$ is ancilla closed and maximal as borrow closed, there can be no ancilla closed RTA between $\operatorname{Aff}(A)$ and $B(A)$ so
 $\operatorname{Aff}(A)$ is a maximal ancilla closed RTA.\qed
\end{proof}

By \cite{aaronsonetal15} we know that for $A$ of order 2,
we have the following maximal ancilla closed RTAs.
\begin{itemize}
 \item The affine mappings,
 \item The odd prime-conservative mappings, that preserve the number of 1s mod $p$, an odd prime,
 \item The parity respecting mappings, which either preserve the number of 1s mod $2$, or invert it.
\end{itemize}
The affine mappings have $k=3$ in Lemma \ref{lemmamaxba_closure} above, the parity respecting $k=2$ and the prime-conservative
mappings $k=1$.

It remains open whether these are the borrow closed maximal RTAs over $A$ of order 2.

For $A$ of order 3, we know that the affine maps are ancilla closed and as they form a maximal borrow closed RTA, they
also form a maximal ancilla closed RTA. Here $k=2$ in the above.
For $A$ of order 4, we do not know which maximal RTA arise with $k=2$.
For $A$ of order 5 or more, we know that $k=2$ arises only for the degenerate RTA $Deg(A)$.

% General result relevant here.
% 
% \begin{lemma}
%  Suppose $G$ is an almost simple group that acts primitively on $64$ or $256$ points.
%  Then $G$ is $A_64$ or $A_256$ respectively.
% \end{lemma}
% 
% 
% \begin{lemma}
%  Suppose $G$ is an almost simple group that acts primitively on $p^n$ points.
%  Then $G$ is .
% \end{lemma}
% \begin{proof}
%  Let $G$ be an almost simple group with socle $T$.
%  Let $G$ act primitively on $m$ points. Then the stabilizer of any point is a
%  maximal subgroup $M$ of $G$ of index $m$, where $M$ does not contain $T$.
%  
%  Thus $M \cap T$ is a subgroup of $T$ also of index $m$. 
%  We know that $m=2^n$ is a prime power.
%  By \cite{guralnick} we know the cases for subgroups of finite simple groups of index a prime power.
% \end{proof}

\section{Conclusion and further work}

We have determined the maximal RTAs, using results from 
permutation group
theory and some generation results.

As we have not been able to construct explicitly an example of a maximal RTA with $i=2$ and $M^{[i]}$ of diagonal or almost simple type,
the conjecture remains that these are not, in fact, possible. We note however that if such examples exist, they will arise
for $A$ of order 8 or more, so will probably not be relevant for any practical reversible computation implementation.

In future work we aim to determine the weight functions as described by \cite{jerabek18} for maximal RTAs, in order to determine whether
they hold some interesting insights.

The results for borrow and ancilla closed RTAs are not as comprehensive, but we see that it should be possible to determine
these in the foreseeable future.
For the ancilla case, we imagine that many of the techniques of \cite{aaronsonetal15} will prove useful.
In the ancilla case, we know all maximal RTA with $k=2$ except for $A$ of order 4.
One of the ongoing tasks is to investigate maximal ancilla closed RTA with $k=1$.

\section{Acknowledgements}

Michael Guidici has helped extensively with understanding primitive permutations groups, for which I thank him greatly.

\bibliographystyle{plain}

\bibliography{20_revcomp}

\end{document}